\documentclass[11pt,reqno]{amsart}
\usepackage{amsmath}
\usepackage{amssymb}
\usepackage{amsthm}
\usepackage{enumerate}
\usepackage[usenames]{color}
\usepackage{eepic,epic}
\usepackage{url} 
\usepackage{verbatim}
\usepackage{epstopdf}
\usepackage{subfig}
\usepackage{multirow}
\usepackage{array}

\usepackage{makecell}
\usepackage[normalem]{ulem}
\usepackage[dvipsnames]{xcolor}
\usepackage{algorithm,algpseudocode}
\usepackage{graphicx}

\newtheorem{thm}{Theorem}[section]
\newtheorem{cor}[thm]{Corollary}
\newtheorem{lem}{Lemma}[section]
\newtheorem{prop}[thm]{Proposition}

\newtheorem{ass}{Assumption}

\theoremstyle{definition}

\theoremstyle{remark}
\newtheorem{rem}[thm]{Remark}
\numberwithin{equation}{section}

\newcommand{\ep}{\epsilon}

\begin{document}
\title[]
{On the convergence of decentralized gradient descent with diminishing stepsize, revisited} 
\author{Woocheol Choi, Jimyeong Kim \\
(Sungkyunkwan University, Mathematics)}
\address{}
\email{choiwc@skku.edu}
\email{jimkim@skku.edu}

\maketitle

\begin{abstract}
Distributed optimization has received a lot of interest in recent years due to its wide applications in various fields.  In this work, we revisit the convergence property of the decentralized gradient descent [A. Nedi{\'c}-A.Ozdaglar (2009)] on the whole space given by
$$
x_i(t+1) = \sum^m_{j=1}w_{ij}x_j(t) - \alpha(t) \nabla f_i(x_i(t)),
$$
where  the stepsize is given as $\alpha (t) = \frac{a}{(t+w)^p}$ with $0< p\leq 1$. Under the strongly convexity assumption on the total cost function $f$ with  local cost functions $f_i$ not necessarily being convex, we show that the sequence converges to the optimizer with rate $O(t^{-p})$ when the values of $a>0$ and $w>0$ are suitably chosen.
\end{abstract}
\section{Introduction}\label{section1}
There has been a significant interest in distributed optimization techniques in recent years since such techniques play an essential role in engineering problems that consist of multiple agents. For example, distributed control \cite{BCM, CYRC}, signal processing \cite{Boyd Gossip, QT}, and machine learning problems \cite{BCN, FCG, RB}. Distributed optimization arises in settings where each agent has their own local cost function and try to find a minimizer of the sum of those local cost functions in a collaborative way that each agent only uses the information from its neighboring agents without a central controller. The problem is written as 
\begin{equation}\label{prob}
\min_{x\in\mathbb{R}^d}~f(x) :=\frac{1}{m} \sum_{i=1}^m f_i (x),
\end{equation}
where $m$ denotes the number of agents and each local cost $f_i : \mathbb{R}^d \rightarrow \mathbb{R}$ is a differentiable function only known to agent $i$. One fundamental algorithm for this problem is the decentralized gradient descent (DGD) due to Nedi\'{c} and Ozdaglar \cite{NO}. This algorithm, consisting of a consensus step based on a communication pattern designed by $w_{ij}\geq0$ and a local gradient step, is stated as follows:
\begin{equation}\label{eq-1-1}
x_i\left(t+1\right) =   \sum^m_{j=1} w_{ij}x_j\left(t\right) -\alpha\left(t\right) \nabla f_i\left(x_i\left(t\right)\right).
\end{equation}
Here ${x}_j(t) \in \mathbb{R}^d$ is the state at time $t \geq 0$ handled by agent $j$ and $\alpha(t) >0$ is the stepsize.  The convergence property of the decentralized gradient descent has been studied in the works \cite{NO, NO2, RNV, I, YLY}. There are also various distributed algorithms containing the distributed dual averaging method \cite{DAW}, consensus-based dual decomposition \cite{FMGP, SJ}, and the alternating direction method of multipliers (ADMM) based algorithms \cite{MJ, SLYW}. We also refer to \cite{XPNK} for a variety of the decentralized optimization algorithms of first order. 

In this paper, we are concerned with the convergence property of the decentralized algorithm \eqref{eq-1-1}. Nedi{\'c}-Ozdaglar \cite{NO} showed that for the algorithm \eqref{eq-1-1} with the stepsize $\alpha(t) \equiv \alpha$, the cost value $f(\cdot)$ at an average of the iterations converges to an $O(\alpha)$-neighborhood of an optimal value of $f$. Ram-Nedi{\'c}-Veeravalli \cite{RNV} proved that the algorithm \eqref{eq-1-1}, involving a projection to a compact set, converges to an optimal point if the stepsize satisfies $\sum_{t=1}^{\infty}\alpha (t) = \infty$ and $\sum_{t=1}^{\infty}\alpha (t)^2 <\infty$. It was extended by Nedi{\'c}-Ozdaglar \cite{NO2} to the whole space and to the stepsize $\alpha (t) = c/\sqrt{t}$. 
In the work of Chen \cite{I}, the algorithm \eqref{eq-1-1} with stepsize $\alpha (t) = c/t^p$ with $0<p<1$ was considered and the convergene rate was achieved as $O(1/t^p)$ for $0<p<1$, $O(\log t/\sqrt{t})$ for $p=1$, and $O(1/t^{1-p})$ for $1/2<p<1$. We mention that all the aforementioned works were established with the gradient bound assumption $\|\nabla f_i \|_{\infty}<\infty$ and convexity assumption on each \mbox{function $f_i$.} 

Recently, Yuan-Ling-Yin \cite{YLY} considered the algorithm \eqref{eq-1-1} without the gradient bound assumption. They showed that if each local cost function is convex and the total cost is strongly convex, then the algorithm \eqref{eq-1-1} with constant stepsize $\alpha (t) \equiv \alpha$ converges exponentially to an $O(\alpha)$-neighborhood of an optimizer $x_*$ of \eqref{prob}. The previous results we mentioned are summarized in Table \ref{known results1}.

In this work, we investigate the convergence property of the algorithm \eqref{eq-1-1} for a general class of non-increasing stepsize $\{\alpha (t)\}_{t \in \mathbb{N}_0}$ given as $\alpha (t) = a/(t+w)^p$ for $a>0$, $w \geq 1$ and $0< p \leq 1$. As in \cite{YLY}, we do not assume the gradient bound. Furthermore, we only assume strong convexity on the total cost function $f$, with cost functions $f_i$ not necessarily being convex. The setting of this assumption is natural in various distributed optimization problem when the local cost of an agent is non-convex near the minimizer of the total cost function. The convergence results of this paper sheds light on choosing suitable values of $a$ and $w$ for  fast convergence of the algorithm \eqref{eq-1-1}.

 \begin{table}[ht]
\centering
\begin{tabular}{|c|c|c|c|c|c|c| }\cline{1-7}
& Cost& Regularity & Learning rate & Error  &Rate & Proj.
\\
\hline
&&&&&&\\[-1em]
\cite{NO}& C  & $\|\nabla f_i \|_{{\infty}}< \infty$ & $\alpha (t) \equiv \alpha$ & $f(\tilde{x}_i (t)) -f_* $& $O\Big( \frac{1}{\alpha t}\Big) + O (\alpha)$ & No
\\
&&&&&&\\[-1em]
\hline
&&&&&&\\[-1em]
\cite{RNV} & C  & $\|\nabla f_i \|_{{\infty}}< \infty$ &\makecell{ $\sum_{t=1}^{\infty}\alpha (t) = \infty$\\ $\sum_{t=1}^{\infty}\alpha(t)^2 < \infty$} & $\|x_i (t) -x_*\|$ & $o(1)$ & Yes
\\
&&&&&&\\[-1em]
\hline
&&&&&&\\[-1em]
\cite{NO2}& C  &$\|\nabla f_i \|_{{\infty}}< \infty$&\makecell{ $\sum_{t=1}^{\infty}\alpha (t) = \infty$\\ $\sum_{t=1}^{\infty}\alpha(t)^2 < \infty$} & $\|x_i (t) -x_*\|$ & $o(1)$ & No
\\
&&&&&&\\[-1em]
\hline
&&&&&&\\[-1em]
\cite{NO2} & C  & $\|\nabla f_i \|_{ {\infty}}< \infty$ & $\alpha (t) = \frac{c}{\sqrt{t}}$ & $\|x_i (t) -x_*\|$ & $o(1)$ & No
\\
&&&&&&\\[-1em]
\hline
&&&&&&\\[-1em]
\cite{I} & C  &$\|\nabla f_i \|_{{\infty}}< \infty$ & $\alpha (t) = \frac{c}{t^{p}}$ & \makecell{$\min_{1 \leq k \leq m}$ \\ $f(x_k (t)) - f_* $}& \makecell{$O(\frac{1}{t^{p}})$ if $0 < p< \frac{1}{2}$ \\  $O( \frac{\log t}{\sqrt{t}})$ if $p = \frac{1}{2}$ \\ $O(\frac{1}{t^{1-p}})$ if $\frac{1}{2}<p <1$} & Yes
\\
&&&&&&\\[-1em]
\hline
&&&&&&\\[-1em]
\cite{YLY} & C& L-smooth & $\alpha (t) \equiv \alpha$ & $f(x_i (t)) - f^*$ & $O(\frac{1}{\alpha t}) + O(\frac{\alpha}{1-\beta})$ & No 
\\
&&&&&&\\[-1em]
\hline
&&&&&&\\[-1em]
\cite{YLY} & SC  & L-smooth & $\alpha (t) \equiv \alpha$ &$ \|x_i (t) - x_*\|$ &$O(e^{-ct}) + O(\frac{\alpha}{1-\beta})$  & No
\\
&&&&&&\\[-1em]
\hline
&&&&&&\\[-1em]
\makecell{This\\work} & SC  & L-smooth & $ \alpha (t) = \frac{a}{(t+w)^p} $& $\|x_i (t) -x_*\|$ & $O(\frac{1}{t^{p}})$  if $0<p \leq 1$& No
\\
\hline
\end{tabular}
\vspace{0.1cm}
\caption{This table summarizes the convergene results for DGD. Here C (resp., SC) means that the total cost function is assumed to be convex (resp., strongly convex). Also $x_*$ is an optimizer of \eqref{prob} and  $\tilde{x}_i(t) = \frac{1}{t}\sum^{t-1}_{s=0}x_i(s)$. We write `Yes' in  `Proj.' if the algorithm \eqref{eq-1-1} with projection is considered.}\label{known results1}
\end{table}

The rest of the paper is organized as follows. In Section \ref{section2}, we state the assumptions and the main results of this paper. Section \ref{section3} is devoted to establishing two sequential inequalities,  which are essentially used in the proofs of the main theorems. In Section \ref{sec4} we prove the uniform boundedness and the consensus error estimate with stepsize in a suitable range. In addition, the range of the stepsize for the uniform boundedness is shown to be almost sharp. In Section \ref{sec5}, the main theorems on the convergence results are proved. In Section \ref{sec-sim} we present the numerical results of the proposed algorithm. In appendix \ref{app-a}, we prove two lemmas on sequential inequalities.

Before ending this section, we state several notations used in this paper. Let $\mathbb{N}_0$ be the set of natural numbers including $0$. For a matrix $A\in \mathbb{R}^{n\times m}$, $a_{ij}$ denotes the $(i,j)$-th entry of $A$. For a row vector $x\in \mathbf{R}^{1\times d}$, $\|x\| =\sqrt{x\cdot x^T}$ denotes the standard Euclidean norm. And we use the notation $\mathbf{x}$ as $\mathbf{x} = [x^T;x^T;\cdots;x^T]^T \in \mathbb{R}^{n\times d}$. In addition, for $X \in \mathbb{R}^{m\times d}$ given by $X = [x^T_1;x^T_2;\cdots;x^T_m]^T$ with a row vector $x_i\in\mathbb{R}^{1\times d}$, we define the $l_2$-norm $\|X\|$ by $\|X\| = \Big(\sum^m_{i=1} \|x_i\|^2\Big)^{1/2}$. Finally, we use the notation $\mathbf{1}$ to denote $[1,1,\cdots,1]^T\in\mathbb{R}^{n\times 1}$.

\section{Preliminaries and Main Results}\label{section2}
In this section, we state the assumptions and the main results of this paper.
We start by making the following standard assumptions on the cost functions in \eqref{prob}.
\begin{ass}\label{LS}
For each $i\in\{1,\cdots m\}$, the local function $f_i$ is $L_i$-smooth for some $L_i>0$, i.e., for any $x, y \in \mathbb{R}^d$ we have
\begin{equation}\label{L-smooth}
\| \nabla f_i(y) - \nabla f_i(x)\| \leq L_i\|y-x\|\quad \forall~x,y \in \mathbb{R}^d.
\end{equation}
It is well-known that \eqref{L-smooth} gives the following estimate
\begin{equation}\label{L-smooth2}
f_i(y) \leq f_i(x) + (y-x)^T\cdot\nabla f_i(x) + \frac{L_i}{2}\|y-x\|^2.
\end{equation}
We set $L = \max_{1\leq i \leq m} L_i$.
\end{ass}
\begin{ass}\label{sc} 
The total cost function $f$ is $\mu$-strongly convex for some $\mu >0$, i.e., 
\begin{equation}\label{strong}
f(y) \geq f(x) + (y-x)^T\cdot\nabla f(x) + \frac{\mu}{2}\|y-x\|^2
\end{equation}
for all $x,y \in \mathbb{R}^d$.
\end{ass}
  
 The communication pattern among agents in \eqref{prob} is determined by an undirected graph $\mathcal{G}=(\mathcal{V},\mathcal{E})$, where each node in $\mathcal{V}$ represents each agent, and each edge $\{i,j\} \in \mathcal{E}$ means $i$ can send messages to $j$ and vice versa. We consider a graph $\mathcal{G}$ satisfying the following assumption.
\begin{ass}\label{graph}
The communication graph $\mathcal{G}$ is undirected and connected, i.e., there exists a path between any two agents.
\end{ass}
We define the mixing matrix $W = \{w_{ij}\}_{1 \leq i,j \leq m}$ as follows. The nonnegative weight $w_{ij}$ is given for each communication link $\{i,j\}\in \mathcal{E},$ where $w_{ij}\neq0$ if $\{i,j\}\in\mathcal{E}$ and $w_{ij} = 0$ if $\{i,j\}\notin\mathcal{E}$. In this paper, we make the following assumption on the mixing matrix $W$. 
\begin{ass}\label{ass-1-1}
The mixing matrix $W = \{w_{ij}\}_{1 \leq i,j \leq m}$ is doubly stochastic, i.e., $W\mathbf{1}=\mathbf{1}$ and $\mathbf{1}^T W = \mathbf{1}^T$. In addition, $w_{ii}>0$ for all $i \in \mathcal{V}$. 
\end{ass}

Under the above two assumptions on the graph, we have the following result (see \cite[Lemma 1]{PN}).
\begin{lem}\label{lem-1-1}
Suppose Assumptions \ref{graph} and \ref{ass-1-1} hold, and let $\beta$ be the spectral norm of the matrix $W - \frac{1}{m} \mathbf{1} \mathbf{1}^T.$  Then there exists a constant $\beta<1$ such that
\begin{equation*}
\sum_{i=1}^m \Big\| \sum_{j=1}^m w_{ij} (x_j - \bar{x}) \Big\|^2 \leq \beta^2 \sum_{i=1}^m \|x_i - \bar{x}\|^2,
\end{equation*}
 where $\bar{x} = \frac{1}{m} \sum_{k=1}^m x_k$ 
 for any $x_i\in \mathbb{R}^{d\times1}$ and $1\leq i \leq m$.
\end{lem}

\subsection{Main result}
Our goal in this paper is to establish the convergence property of the decentralized gradient descent \eqref{eq-1-1} for a general class of non-increasing stepsize $\{\alpha (t)\}_{t \in \mathbb{N}_0}$ given as $\alpha (t) = a/(t+w)^p$ for $a>0$, $w \geq 1$ and $0\leq p \leq 1$. Before stating our main results, we define some notations and constants to be used.

Let $x_* = \arg\min_{x\in\mathbb{R}^d} f(x)$ be the optimal point, whose existence is guaranteed by Assumption \ref{sc}, and  $D = \max_{1 \leq i \leq n} \|\nabla f_i (x_*)\|$. We regard $x_i(t)$ as a row vector in $\mathbb{R}^{1\times d}$, and define the variable  $\mathbf{x}(t)\in \mathbb{R}^{m\times d}$ by
\begin{equation}\label{eq-2-4}
\mathbf{x}(t)=\left(x_1\left(t\right)^T, \cdots, x_m\left(t\right)^T\right)^T. 
\end{equation}
We also define $\bar{\mathbf{x}}(t) \in \mathbb{R}^{m \times d}$ and $\mathbf{x}_* \in \mathbb{R}^{m\times d}$ by
\begin{equation*}
\bar{\mathbf{x}}(t) = \left(\bar{x}\left(t \right)^T, \cdots, \bar{x}\left( t\right)^T \right)^T\quad \textrm{and}\quad \mathbf{x}_* = \left( x_*^T, \cdots, x_*^T \right)^T,
\end{equation*}
where $\bar{x}(t) = \frac{1}{m} \sum_{k=1}^m x_k(t)$. We denote by the constant $R>0$ the uniform upper bound for the quantities $\left\|\bar{\mathbf{x}}(t) - \mathbf{x}_*\right\|$ and $\left\|\mathbf{x}(t)- \bar{\mathbf{x}}(t)\right\|$ for $t\geq 0$. The existence of such a constant $R>0$ will be proved in Section \ref{sec4}. For notational convenience, we also define the following constants:
\begin{equation*}
\begin{split}
&d= 2LR + \sqrt{m}D, \ R_0 = \|\mathbf{x}(0) - \overline{\mathbf{x}}(0)\| + \frac{d\alpha (0)}{1-\beta},  \\
& \eta= \frac{\mu L}{\mu +L}, \ Q_1 = \Big( \frac{w+1}{w}\Big)^{2p}, \ Q_0 =\frac{Q_1Ld2^p}{\eta}.
\end{split}
\end{equation*}

In order to investigate the convergence properties of the sequence $\left\{x_i (t)\right\}_{t \in \mathbb{N}_0}$ generated by \eqref{eq-1-1}, we split the error $\|\mathbf{x}(t) - \mathbf{x}_*\|$ using the following equality: 
\begin{equation*}
\begin{split}
 \|\mathbf{x}(t) -\mathbf{x}_*\|^2 &=   \sum^m_{i=1} \|x_i(t)-x_*\|^2  
 \\
 & = \sum_{i=1}^m \|x_i (t) - \bar{x}(t)\|^2 + \sum_{i=1}^m \|\bar{x}(t) -x_*\|^2
 \\
  &= \|\mathbf{x}(t) -\bar{\mathbf{x}}(t)\|^2 +  \|\bar{\mathbf{x}}(t)-\mathbf{x}_*\|^2,
\end{split}
\end{equation*}
whose right hand side consists of the consensus error and the distance between the average $\bar{x}(t)$ and the optimal point $\mathbf{x}_*$.
The following theorem provides a sharp estimate on the consensus error $\|\mathbf{x}(t) - \bar{\mathbf{x}}(t)\|$.
\begin{thm}\label{cor-3-4}
Suppose that  Assumptions \ref{LS} - \ref{ass-1-1}  hold, and  $\{\alpha(t)\}_{t\in\mathbb{N}_0}$ is non-increasing sequence satisfying
$$
\alpha(0) \leq \min\left\{\frac{2}{\mu+L},\frac{\eta (1-\beta)}{L ( \eta + L )}\right\}.
$$
Then, for all $t\geq0$ we have
\begin{equation}\label{3-12}
\|\mathbf{x}(t) - \bar{\mathbf{x}}(t)\| \leq \frac{d}{1-\beta} \alpha ([t/2])  +\beta^t \|\mathbf{x}(0) - \bar{\mathbf{x}}(0)\| + \frac{\beta^{t/2}d}{1-\beta} \alpha (0).
\end{equation}
\end{thm}
This theorem implies that the consensus is achieved with a rate depending on the stepsize $\alpha(t)$. Next we establish the convergence results for $\|\bar{\mathbf{x}}(t) -\mathbf{x}_*\|$. Firstly we state the result for $p \in (0,1)$.
\begin{thm}\label{main}
Suppose that Assumptions \ref{LS} - \ref{ass-1-1} hold. Let $p\in (0,1)$ and assume that $\alpha(t) = \frac{a}{(t+w)^p}$ with  constants $a>0$ and $w\geq 1$ satisfying  
\begin{equation}\label{eq-3-40}
\alpha(0) = \frac{a}{w^p}\leq \min\left\{\frac{2}{\mu+L},\frac{\eta (1-\beta)}{L ( \eta + L )}\right\}.
\end{equation}
Then we have
\begin{equation}\label{eq-2-40}
\begin{split}
 \|\bar{\mathbf{x}}(t) -\mathbf{x}_*\| \leq    \frac{(Q_0 \sqrt{e}) \, a}{1-\beta}  \cdot \Big( [t/2] + w-1\Big)^{-p} +  Y_1 (t) +Y_2 (t) + Y_3(t),
\end{split}
\end{equation}
where
\begin{equation*}
\begin{split}
Y_1 (t) & = e^{-\sum_{s=0}^{t-1} \frac{\eta a}{(s+w)^p}} \|\mathbf{x}(0) -\mathbf{x}_*\|
\\
Y_2 (t) & = \frac{Q_0  }{1-\beta} e^{-\frac{\eta a}{2}\cdot \frac{t}{(t+w)^p}} \sum_{s=1}^{[t/2]-1} \frac{a^2}{(w+s)^{2p}}
\\
Y_3(t)& = \frac{aLR_0}{(1-\beta)w^p}\cdot \bigg( \frac{e^{ -\frac{\eta a}{2 } \frac{(t-1)}{(t+w)^p}}}{1-\sqrt{\beta}} + \frac{\sqrt{\beta}^{[(t-1)/2]}}{1-\sqrt{\beta}}\bigg).
\end{split}
\end{equation*}
Here $[z]$ for $z \in \mathbb{R}$ denotes the largest integer less or equal to $z$.
It is easy to see that for any fixed $N>0$, there exists a constant $C_N>0$ independent of $t\geq 0$ such that 
\begin{equation*}
Y_1 (t) + Y_2 (t) + Y_3(t) \leq C_N t^{-N}.
\end{equation*}
\end{thm}
In the estimate \eqref{eq-2-40}, we see that if $a/w^p$ is large, then the exponential terms in $Y_1(t), Y_2(t)$ and $Y_3(t)$ decrease fast but the first term in \eqref{eq-2-40} becomes large. On the other hand, if we choose small $a/w^p >0$, the first term in \eqref{eq-2-40} gets small, but the exponential terms decrease slowly. These imply that depending on the choice of $a$ and $w$ we can have faster convergence at the expense of a large error in the early stage, or the other way around.   

Next we state the result for the case $p=1$.
\begin{thm}\label{thm-2-4}
Suppose that Assumptions \ref{LS} - \ref{ass-1-1} hold. Let $p=1$ and assume that $\alpha(t) = \frac{a}{(t+w) }$ with positive constants $a>0$ and $w\geq 1$ satisfying  
\begin{equation}\label{eq-3-41}
\alpha(0) = \frac{a}{w }\leq \min\left\{\frac{2}{\mu+L},\frac{\eta (1-\beta)}{L ( \eta + L )}\right\}.
\end{equation}
 If $a>\frac{2}{\eta} $, then we have 
\begin{equation}\label{eq-2-41}
\begin{split}
&\|\bar{\mathbf{x}}(t) -\mathbf{x}_*\|   \leq \Big( \frac{w}{t+w}\Big)^{\eta a}\|\mathbf{x}(0) - \mathbf{x}_*\|  + \frac{2\sqrt{e}Q_0}{(1-\beta) }\cdot\frac{a}{(t+w+1)}  +Y_4(t),
\end{split}
\end{equation}
where
\begin{equation*}
Y_4(t) = \frac{aLR_0}{1-\beta}\left[ \frac{w(w+1)^{l-1}}{1-e^{l/(w+1)}\sqrt{\beta}}\cdot \frac{1}{(t+w)^{l+1}} + \frac{\sqrt{\beta}^t}{(t-1+w)}\right].
\end{equation*}
Here $l$ is any value such that $0 \leq l \leq a-1$ and $e^{l/(w+1)} \sqrt{\beta} <1$.
\end{thm}
Similarly as in Theorem \ref{main}, we find that if $a/w$ is small, then the second term in the right hand side of \eqref{eq-2-41} is small, but the first term decays slowly. On the contrary, if $a/w$ is large, then the second term is significant while the first term decays fast.

\begin{rem}
The bound $\frac{\eta (1-\beta)}{L ( \eta + L )}$ of \eqref{eq-3-40} and \eqref{eq-3-41} is only used to prove the uniform boundedness of the sequences $x_i (t)$ for $1 \leq i \leq n$ and $t \in \mathbb{N}$. If there is a prior guarantee that the sequence $x_i (t)$ are uniformly bounded, then the bound $\frac{2}{\mu + L}$ of \eqref{eq-3-40} and \eqref{eq-3-41} is sufficient for the convergence estimate of the above theorem. 

Interestingly, the bound $\frac{\eta (1-\beta)}{L ( \eta + L )}$ is quite sharp for certain examples as revealed in Subsection \ref{sec4-1}. In our example, we consider a strongly convex total cost function which consists of convex and non-convex  local cost functions. The uniform boundedness might be achieved for larger values of $\alpha(0)$ under some suitable conditions of the cost functions, e.g. convexity on each local function (see \cite{YLY}). 
\end{rem}

\section{Sequential Estimates}\label{section3}
In this section we derive sequential estimates of $\|\bar{\mathbf{x}}(t) - \mathbf{x}_*\|$ and $\|\mathbf{x}(t) - \bar{\mathbf{x}}(t)\|$, which are  essential for establishing our main results. 
By summing up \eqref{eq-1-1} for $1\leq i \leq n$, we have
\begin{equation}\label{scheme3}
\bar{x}(t+1) = \bar{x}(t) - \frac{\alpha(t)}{m}\sum^m_{i=1} \nabla f_i(x_i(t)).
\end{equation}
Thanks to \eqref{eq-2-4}, we may write \eqref{eq-1-1} in a compact form as
\begin{equation}\label{scheme2}
\mathbf{x}(t+1) =W \mathbf{x}(t) - \alpha(t) \nabla f(\mathbf{x}(t)),
\end{equation}
where 
$$
 \nabla f(\mathbf{x}(t)) = \left(\nabla f_1( x_1(t))^T, \nabla f_2( x_2(t))^T, \cdots, \nabla f_m( x_m(t))\right)^T.
$$
In the following lemma, we obtain a bound of $\|\bar{\mathbf{x}}(t+1)-\mathbf{x}_*\|$ in terms of $\|\bar{\mathbf{x}}(t)-\mathbf{x}_*\|$ and $\|\mathbf{x}(t) - \bar{\mathbf{x}}(t)\|$.
\begin{lem}\label{est1}
Suppose that Assumptions \ref{LS} - \ref{ass-1-1} hold. If $\{\alpha(t)\}_{t\in\mathbb{N}_0}$ is non-increasing stepsize satisfying $\alpha(0)\leq \frac{2}{\mu+L},$ then we have
\begin{equation}\label{estimate1}
\|\mathbf{\bar{x}}(t+1)-\mathbf{x}_*\| \leq \left(1-{\eta \alpha(t)} \right)\| \mathbf{\bar{x}}(t)-\mathbf{x}_*\| + L\alpha(t) \| \mathbf{x}(t)-\mathbf{\bar{x}}(t)\|.
\end{equation}
\end{lem}
\begin{proof}
Using the triangle inequality, we deduce
\begin{equation}\label{eq-3-3}
\begin{split}
&\|\bar{x}(t+1)-x_*\| 
\\
&= \left\| \bar{x}(t)-x_*- \frac{\alpha(t)}{m}\sum^m_{i=1} \nabla f_i(x_i(t))\right\| \\
&\leq \left\| \bar{x}(t)-x_*- \frac{\alpha(t)}{m}\sum^m_{i=1} \nabla f_i(\bar{x}(t))\right\|  + \frac{\alpha(t)}{m}\sum^m_{i=1} \|\nabla f_i(\bar{x}(t)) - \nabla f_i(x_i(t))\|.
\end{split}
\end{equation}
We first estimate the first term in the last inequality of \eqref{eq-3-3}.
Note that
\begin{align*}
&\| \bar{x}(t)-x_*-  \alpha(t) \nabla f (\bar{x}(t))\|^2 \\
&=\| \bar{x}(t)-x_*\|^2 -  2\alpha(t)( \bar{x}(t)-x_*)^T\cdot\nabla f(\bar{x}(t))+  {\alpha(t)^2}\| \nabla f(\bar{x}(t)\|^2.
\end{align*}
By the assumptions \ref{LS} and \ref{sc}, the cost function $f$ is $L$-smooth and $\mu$-strongly convex. Thus we have the following inequality (see e.g., \cite[Lemma 3.11]{Bu}):
$$
(\bar{x}(t)-x_*)^T\cdot(\nabla f (\bar{x}(t))-\nabla f (x_*)) \geq \frac{\mu L}{\mu+L}\|\bar{x}(t)-x_*\|^2 + \frac{1}{\mu+L}\|\nabla f (\bar{x}(t)) - \nabla f  (x_*)\|^2.
$$
Inserting this into the above equality yields
\begin{align*}
&\| \bar{x}(t)-x_*-  {\alpha(t)}   \nabla f (\bar{x}(t))\|^2 \\
&\leq \bigg(1-\frac{2\mu L\alpha(t)}{\mu+L} \bigg)\| \bar{x}(t)-x_*\|^2  + \bigg(\alpha(t)^2-\frac{2\alpha(t)}{\mu+L}\bigg)  \|\nabla f (\bar{x}(t)) \|^2.
\end{align*}
Using $\alpha(t)\leq \frac{2}{\mu+L}$ and  letting $\eta = \frac{\mu L}{\mu +L}$, we have
\begin{equation}\label{eq-3-1}
\begin{split}
\left\| \bar{x}(t)-x_*- \frac{\alpha(t)}{m}\sum^m_{i=1} \nabla f_i(\bar{x}(t))\right\|^2  &\leq \bigg(1-\frac{2\mu L\alpha(t)}{\mu+L} \bigg)\| \bar{x}(t)-x_*\|^2 
\\
&\leq \Big(1-\eta\alpha (t)\Big)^2\| \bar{x}(t)-x_*\|^2.
\end{split}
\end{equation}
Next we bound the second term in the last inequality of $\eqref{eq-3-3}$
 
By $L_i$-smoothness of $f_i$, we have
\begin{equation}\label{eq-3-2}
\sum^m_{i=1}\|\nabla f_i(\bar{x}(t)) - \nabla f_i(x_i(t))\|\leq L\sum^m_{i=1}\|\bar{x}(t)-x_i(t)\|  \leq L\sqrt{m}\| \mathbf{x}(t)-\mathbf{\bar{x}}(t)\|,
\end{equation}
where $L=\max_{1\leq i \leq n} L_i$. 
Putting \eqref{eq-3-1} and \eqref{eq-3-2} into \eqref{eq-3-3}, we obtain
\begin{align*}
\|\bar{x}(t+1)-x_*\| &\leq (1-\eta\alpha(t) )\| \bar{x}(t)-x_*\| + \frac{L\alpha(t)}{\sqrt{m}} \| \mathbf{x}(t)-\mathbf{\bar{x}}(t)\|.
\end{align*}
 Using the fact $\|\mathbf{\bar{x}}(t)-\mathbf{x}_* \|= \sqrt{m}\|\bar{x}(t)-x_*\|$, the above inequality gives
\begin{equation*}
\|\mathbf{\bar{x}}(t+1)-\mathbf{x}_*\| \leq (1-\eta\alpha(t) )\| \mathbf{\bar{x}}(t)-\mathbf{x}_*\| + L\alpha(t) \| \mathbf{x}(t)-\mathbf{\bar{x}}(t)\|.
\end{equation*}
The proof is finished.
\end{proof}

Next we establish a bound of $\|\mathbf{x}(t+1)-\mathbf{\bar{x}}(t+1)\|$ in terms of  $\|\mathbf{x}(t)-\mathbf{\bar{x}}(t)\|$ and $\|\bar{\mathbf{x}}(t)-\mathbf{x}_*\|$.  
\begin{lem}\label{est2}
Suppose that Assumptions \ref{LS} - \ref{ass-1-1} hold. Then we have
\begin{equation}\label{estimate2}
\begin{split}
&\|\mathbf{x}(t+1)-\mathbf{\bar{x}}(t+1)\|\leq  (\beta +  L\alpha(t))\| \mathbf{x}(t)- \mathbf{\bar{x}}(t) \|+   L \alpha(t)\| \mathbf{\bar{x}}(t)-\mathbf{x}_*\| +\sqrt{m}D \alpha(t),
\end{split}
\end{equation}
where $D = \max_{1 \leq i \leq n} \|\nabla f_i (x_*)\|$.
\end{lem}

\begin{proof}
Combining \eqref{scheme2} and Lemma \ref{lem-1-1} yields
\begin{equation}\label{eq-3-4}
\begin{split}
\|\mathbf{x}(t+1)-\mathbf{\bar{x}}(t+1)\| &= \Big\| W\mathbf{x}(t) - \alpha(t) \nabla f(\mathbf{x}(t))  - \mathbf{\bar{x}(t)} + \frac{\alpha(t)}{m}\big(\mathbf{1}\mathbf{1}^T \nabla f(\mathbf{x}(t))\big)\Big\| \\
&\leq \| W(\mathbf{x}(t)-\mathbf{\bar{x}}(t))\| + \alpha(t)\Big\|(\text{I}-\frac{1}{m}\mathbf{1}\mathbf{1}^T)\nabla f(\mathbf{x}(t))\Big\| \\
&\leq \beta \|\mathbf{x}(t)-\mathbf{\bar{x}}(t)\|  + \alpha(t)\| \nabla f(\mathbf{x}(t))\|.
\end{split}
\end{equation}
Here I is $n\times n$ identity matrix. By $L_i$-smoothness of $f_i$, we have
\begin{equation}\label{eq-3-5}
\begin{split}
\|\nabla f_i (x_i (t)) \|&\leq \|\nabla f_i (x_i (t)) - \nabla f_i (\bar{x}(t))\| + \|\nabla f_i (\bar{x}(t)) - \nabla f_i (x_*)\| + \|\nabla f_i (x_*)\|
\\
&\leq L_i \|x_i (t) - \bar{x}(t)\| + L_i \|\bar{x}(t) -x_*\| + D.
\end{split}
\end{equation}
We recall that $L=\max_{1\leq i \leq n} L_i$ and use the triangle inequality to bound $\|\nabla f(\mathbf{x}(t))\|$ as
\begin{equation}\label{eq-3-6}
\begin{split}
\|\nabla f (\mathbf{x}(t))\|&  = \left( \sum_{i=1}^m \|\nabla f_i (x_i (t))\|^2 \right)^{1/2} 
\\
&\leq  L\| \mathbf{x}(t)- \mathbf{\bar{x}}(t) \| + L \| \mathbf{\bar{x}}(t)-\mathbf{x}_*\| +\sqrt{m} D.
\end{split}
\end{equation} Inserting \eqref{eq-3-6} into \eqref{eq-3-4}, we get
\begin{equation*}
\|\mathbf{x}(t+1)-\mathbf{\bar{x}}(t+1)\| \leq  (\beta +L \alpha(t))\| \mathbf{x}(t)- \mathbf{\bar{x}}(t) \|+  L\alpha(t)  \| \mathbf{\bar{x}}(t)-\mathbf{x}_*\|   +\sqrt{m}D \alpha(t).
\end{equation*}
The proof is finished.
\end{proof}

\section{Uniform boundedness and the consensus estimate}\label{sec4}
In this section we establish the uniform boundedness of the sequences $\{\|\bar{\mathbf{x}}(t) - \mathbf{x}_*\|\}_{t \in \mathbb{N}_0}$ and $\{\|\mathbf{x}(t)- \bar{\mathbf{x}}(t)\|\}_{t \in \mathbb{N}_0}$ , and present the proof of Theorem \ref{cor-3-4}.   We first show that the sequence $\{x_i (t)\}_{t \in \mathbb{N}_0}$ are uniformly bounded under a condition on $\alpha (0) >0$. 
\begin{thm}\label{bound}
Suppose that  Assumptions \ref{LS} - \ref{ass-1-1}, \eqref{estimate1} and \eqref{estimate2} hold. If we assume a diminishing sequence $\{\alpha(t)\}_{t\in\mathbb{N}_0}$ satisfies $\alpha (0) < \frac{\eta (1-\beta)}{L ( \eta + L )}$ and set a finite value $R>0$ as
\begin{equation*}
R = \max\left\{ \|\bar{\mathbf{x}}(0) - \mathbf{x}_*\|,~\frac{L}{\eta}\|\mathbf{x}(0)- \bar{\mathbf{x}}(0)\|,~\frac{\sqrt{m} D\alpha(0)}{\eta (1-\beta)/L  - (\eta +L) \alpha (0)}\right\},
\end{equation*} 
then we have
\begin{equation}\label{eq-2-30}
\|\bar{\mathbf{x}}(t)-\mathbf{x}_* \|  \leq R \quad and \quad \|\mathbf{x}(t)-\bar{\mathbf{x}}(t) \| \leq \frac{\eta R}{L}<R, \quad \forall t \geq 0.
\end{equation}
\end{thm}
\begin{proof}
We argue by an induction. By the definition of $R$, we have 
\begin{equation*} 
\|\bar{\mathbf{x}}(0)-\mathbf{x}_* \|  \leq R \quad and \quad \|\mathbf{x}(0)-\bar{\mathbf{x}}(0) \| \leq \frac{\eta}{L}R.
\end{equation*}
Next we  assume that the following inequalities
\begin{equation}\label{eq-3-7}
\|\bar{\mathbf{x}}(t)-\mathbf{x}_* \|  \leq R \quad and \quad \|\mathbf{x}(t)-\bar{\mathbf{x}}(t) \| \leq \frac{\eta}{L}R
\end{equation}
holds true for some fixed $t \geq 0$. Then, plugging these bounds in \eqref{estimate1}, we get
\begin{align*}
\| \bar{\mathbf{x}}(t+1) - x_*\| \leq \Big(1-{\eta \alpha(t)}\Big)R + L\alpha(t)\Big(\frac{\eta}{L}\Big)R = R.
\end{align*}
Combining \eqref{eq-3-7} and \eqref{estimate2}, we have
\begin{equation*}
\|\mathbf{x}(t+1) - \bar{\mathbf{x}}(t+1)\| \leq (\beta + L\alpha(t)) \Big(\frac{\eta}{L}\Big)R + L\alpha(t)R+\sqrt{m} D\alpha(t).
\end{equation*}
Using the definition of $R$ and the fact that $\alpha (t)$ is non-increasing, we deduce 
\begin{equation*}
\begin{split}
\|\mathbf{x}(t+1) - \bar{\mathbf{x}}(t+1)\| &\leq \Big[(\beta + L\alpha(t)) \Big(\frac{\eta}{L}\Big) + L\alpha(t)  \Big] R + \sqrt{m} D \alpha (t)
\\
& = \left[ \frac{\beta \eta}{L} + ( {\eta} + L ) \alpha(t) \right]R + \sqrt{m}D \alpha (t)\\
&\leq  \left[ \frac{\beta \eta}{L} + ( {\eta} + L )   \alpha(0)\right]R + \sqrt{m}D \alpha (0).
\end{split}
\end{equation*}
Notice that the condition $\alpha (0) < \frac{\eta (1-\beta)}{L ( \eta + L )}$  implies $\frac{(1-\beta)\eta}{L} -(\eta +L)\alpha(0)>0$. Thus $R>0$ is well-defined and the following inequality follows:
\begin{equation*}
\begin{split}
\left[ \frac{\beta \eta}{L} + ( {\eta} + L ) (t) \alpha(0)\right]R + \sqrt{m}D \alpha (0) &\leq \frac{\eta}{L} R.
\end{split}
\end{equation*}
This completes the induction, and so the proof is done.
\end{proof}
\begin{rem}
We will show that the range of $\alpha (0)$ is almost sharp for large $L>0 $ and small $\mu >0$ in Subsection \ref{sec4-1}.
\end{rem}
By the uniform bound result of Theorem \ref{bound}, we now prove  Theorem \ref{cor-3-4}.
\begin{proof}[Proof of Theorem \ref{cor-3-4}]
Note that the inequality \eqref{3-12} holds true trivially for $t=0$. Hence we consider the case $t\geq1$.
By Lemma \ref{est2} and Theorem \ref{bound}, we have
\begin{equation}\label{eq-3-17}
 \| \mathbf{x}(t+1) - \bar{\mathbf{x}}(t+1)\| \leq \beta \|\mathbf{x}(t) - \bar{\mathbf{x}}(t)\| + d\alpha(t),
\end{equation}
where $d = 2LR +\sqrt{m}D$. Using this iteratively gives
\begin{equation}\label{eq-3-8}
 \|\mathbf{x}(t) - \bar{\mathbf{x}}(t) \| \leq \beta^t \|\mathbf{x}(0) - \bar{\mathbf{x}}(0)\| + \sum_{s=0}^{t-1} \beta^{t-1-s} d \alpha (s).
\end{equation}
From this we easily see that \eqref{3-12} holds true for $t=1$. For $t\geq 2$, we estimate
\begin{equation*}
\begin{split}
\sum_{s=0}^{t-1} \beta^{t-1-s}\alpha (s) & = \sum_{s=0}^{[t/2]-1} \beta^{t-1-s} \alpha (s) + \sum_{s=[t/2]}^{t-1} \beta^{t-1-s}\alpha (s) 
\\
&\leq \alpha (0) \sum_{s=0}^{[t/2]-1} \beta^{t-1-s} + \alpha ([t/2]) \sum_{s=0}^{\infty} \beta^{s}
\\
&\leq \alpha (0) \frac{\beta^{t/2}}{1-\beta} + \alpha ([t/2]) \frac{1}{1-\beta}.
\end{split}
\end{equation*}
Using these estimates in \eqref{eq-3-8}, it follows that
\begin{equation*}
 \|\mathbf{x}(t) - \bar{\mathbf{x}}(t)\| \leq \frac{d}{1-\beta} \alpha ([t/2])  + \beta^t \|\mathbf{x}(0) - \bar{\mathbf{x}}(0)\| + \frac{\beta^{t/2}d}{1-\beta} \alpha (0) .
\end{equation*}
The proof is done.
\end{proof}

\subsection{Sharpness of the range of Theorem \ref{bound}}\label{sec4-1}
The uniform boundedness result of Theorem \ref{bound} was also achieved in \cite{YLY} when the stepsize is constant  and each function is convex. Precisely, it was shown that the uniform boundedness is guaranteed when $\alpha (0) \leq 1/L$, which is less restricted than the condition $\alpha (0) < \frac{\eta (1-\beta)}{L (\eta + L)}$ given in Theorem \ref{bound}. On the other hand, the assumptions of Theorem \ref{bound} allow each function to be nonconvex as long as the total cost remains to be strongly convex. Also, the stepsize may be time-varying. 

To verify the sharpness of the range $\alpha(0) < \frac{\eta(1-\beta)}{L(\mu+L)}$ of the assumptions in Theorem \ref{bound}, we construct an example with the following function:
\begin{equation*}
f_1 (x) = a_1 x^2 \quad \textrm{and}\quad f_2 (x) = - a_2 x^2, \quad x \in \mathbb{R},
\end{equation*}
where $a_1$ and $a_2$ are positive values satisfying $a_1 -a_2 >0$. Then the total cost $f= 1/2(f_1 +f_2)$ is strongly convex even though the local cost $f_2$ is non-convex.  We take a value $\gamma \in (0,1/2]$ and set a doubly stochastic matrix $W$ by
\begin{equation*}
W = \begin{pmatrix} 1 - \gamma& \gamma \\ \gamma & 1-\gamma \end{pmatrix}.
\end{equation*}
Let $x_1 (t)$ and $x_2 (t)$ be the variables which are only known to agents $1$ and $2$, respectively. 

\begin{lem}\label{lem-b-1}
If $\alpha > \frac{\gamma  (a_1 -a_2)}{2 a_1 a_2}$, then the sequence $\{(x_1 (t), x_2 (t))\}_{t\geq 0}$ generated by \eqref{eq-1-1} with any initial data $(x_1(0), x_2(0)) \in (\mathbb{R}\setminus \{0\})^2$ diverges.
\end{lem}
\begin{proof}
For this example, the decentralized gradient descent \eqref{eq-1-1} is written as 
\begin{equation*}
\begin{split}
x_1 (t+1)& = (1-\gamma ) x_1 (t) + \gamma  x_2 (t) - 2 \alpha  a_1 x_1 (t)
\\
x_2 (t+1)& =\gamma  x_1 (t) + (1-\gamma ) x_2 (t) + 2 \alpha a_2 x_2 (t).
\end{split}
\end{equation*}
This can be written in a vector form as follows
\begin{equation}\label{eq-4-25}
\begin{pmatrix} x_1 (t+1) \\ x_2 (t+1) \end{pmatrix} = M \begin{pmatrix} x_1 (t) \\ x_2 (t) \end{pmatrix},
\end{equation}
where
\begin{equation*}
M = \begin{pmatrix} 1- \gamma  - 2 \alpha a_1 & \gamma  \\ \gamma  & 1-\gamma  + 2\alpha a_2 \end{pmatrix}.
\end{equation*}
The sequence $(x_1 (t), x_2 (t))$ of \eqref{eq-4-25} diverges if $M$ has an eigenvalue larger than one. 
The eigenvalues of the matrix $M$ are given by $\lambda \in \mathbb{R}$ solving the following equation
\begin{equation*}
\lambda^2 - 2 (1- \gamma  + (a_2 -a_1) \alpha ) \lambda + (1-\gamma  -2 a_1 \alpha) (1-\gamma  + 2a_2 \alpha) - \gamma ^2 = 0.
\end{equation*}
The solutions are
\begin{equation*}
(1-\gamma ) + (a_2 -a_1) \alpha \pm \sqrt{(a_1 +a_2)^2 \alpha^2 + \gamma ^2}.
\end{equation*}
This formula allows us to show that the largest eigenvalue is larger than $1$ if $\alpha > \frac{\gamma  (a_1 -a_2)}{2a_1 a_2}$, which implies the sequence $\{(x_1(t),x_2(t)\}$ diverges. In fact, it is checked in the following way
\begin{equation*}
\begin{split}
&(1 - \gamma ) + (a_2 -a_1) \alpha + \sqrt{(a_1 +a_2)^2 \alpha^2 + \gamma ^2} > 1
\\
& \Leftrightarrow \sqrt{(a_1 + a_2)^2 \alpha^2 + \gamma^2} > \gamma  + (a_1 -a_2) \alpha
\\
&\Leftrightarrow (a_1 +a_2)^2 \alpha^2 + \gamma^2 >  \gamma^2 + 2  \gamma (a_1 -a_2) \alpha + (a_1 -a_2)^2 \alpha^2
\\
&\Leftrightarrow 4a_1 a_2 \alpha^2 > 2  \gamma (a_1 -a_2) \alpha
\\
& \Leftrightarrow \alpha > \frac{ \gamma (a_1 -a_2)}{2a_1 a_2}.
\end{split}
\end{equation*}
The proof is done.
\end{proof}

Now we show that the range $\alpha(0)<\frac{\eta(1-\beta)}{L(\mu +L)}$ is almost sharp for large $L$ and small $\mu$ for the example considered in Lemma \ref{lem-b-1}.
\begin{cor}\label{cor-4-2}
If $\alpha (0) >\frac{\mu(1-\beta)}{L (L-2 \mu)}$, then the sequences $\{x_i (t)\}_{t \in \mathbb{N}_0}$ may diverge for $i=1,2,\cdots,m$. This implies that  the condition $\alpha(0)<\frac{\eta(1-\beta)}{L(\mu +L)}$  is sharp in the sense that
\begin{equation*}
\lim_{L \rightarrow \infty} \frac{\eta (1-\beta)}{L (\eta + L)}\Big/ \frac{\mu(1-\beta)}{L (L-2 \mu)} = 1\quad \textrm{and}\quad \lim_{\mu \rightarrow 0} \frac{\eta (1-\beta)}{L (\eta + L)}\Big/ \frac{\mu(1-\beta)}{L (L-2 \mu)} = 1. 
\end{equation*}
\end{cor}
\begin{proof}
In the setting of Lemma \ref{lem-b-1}, we let $a_1 = \frac{L}{2}$ and $a_2 = \frac{L}{2}- \mu$ with a value $\mu >0$ and a large number $L>0$. Then $f_1$ and $f_2$ are $L$-smooth functions and $f(x) = (\mu/2)x^2$ is $\mu$-strongly convex. Also we have $\beta = 1- 2\gamma$ in Lemma \ref{lem-1-1}. Then the condition on $\alpha >0$ of Lemma \ref{lem-b-1} is written as 
\begin{equation}\label{eq-b-1}
\alpha > \frac{2\mu \gamma}{L (L - 2\mu)}.
\end{equation}
On the other hand, the condition of Theorem \ref{bound} is written as
\begin{equation*}
\left( \frac{\mu L}{\mu +L} + L  \right) \alpha < \frac{\mu L}{\mu +L} \cdot \frac{2\gamma}{L},
\end{equation*}
which is equivalent to  
\begin{equation}\label{eq-b-2}
\alpha < \frac{2 \mu \gamma}{\Big( \frac{\mu L}{\mu +L} + L   \Big) (\mu +L)}.
\end{equation}
Thus the condition \eqref{eq-b-1} is sharp in the sense that the right hand sides of \eqref{eq-b-1} and \eqref{eq-b-2} are very close when $L$ is sufficiently large, which also can be seen by the limit
\begin{equation*}
\lim_{L \rightarrow \infty}  \frac{2\mu \gamma}{L (L-2\mu)}\Big/\frac{2 \mu \gamma}{\Big( \frac{\mu L}{\mu +L} + L \Big) (\mu +L)}=1.
\end{equation*}
Similarly, the range is sharp for sufficiently small $\mu$, in view of the following limit
\begin{equation*}
\lim_{\mu \rightarrow 0}  \frac{2\mu \gamma}{L (L-2\mu)}\Big/\frac{2 \mu \gamma}{\Big( \frac{\mu L}{\mu +L} + L \Big) (\mu +L)}=1.
\end{equation*}
The proof is done.
\end{proof}

\section{Convergence Analysis}\label{sec5}
In this section, we give the proofs of  Theorems \ref{main} and \ref{thm-2-4}. In Section \ref{sec4}, we obtained the uniform boundedness and a sharp estimate on the consensus error $\|\mathbf{x}(t) - \bar{\mathbf{x}}(t)\|$. Here we will obtain a sharp estimate  of $\|\bar{\mathbf{x}}(t) - \mathbf{x}_*\|$ using Theorem  \ref{cor-3-4} and Theorem \ref{bound} together with Lemma  \ref{est1} and the following proposition.
\begin{prop}\label{convergencerate}
 Let $ p \in (0,1]$ and $q>0$. Take $C_1 >0$ and $w \geq 1$ such that $C_1/w^p <1$. 
Suppose that the sequence $\{A(t)\}_{t\in\mathbb{N}}$ satisfies
\begin{equation}\label{4-4}
A(t) \leq \bigg(1-\frac{C_1}{(t+w-1)^p}\bigg) A(t-1) + \frac{C_2}{(t+w-1)^{p+q}} \quad \text{for all $t\geq1$}.
\end{equation}
Set $Q = \Big( \frac{w+1}{w}\Big)^{p+q}$. Then $A(t)$ satisfies the following bound.

\medskip 

\noindent Case 1. If $p<1$, then we have
\begin{equation*}
 A(t) \leq   \delta \cdot([t/2]+w-1)^{-q} + \mathcal{R}(t),
\end{equation*}
where  $\delta =  \frac{QC_2}{C_1} e^{\frac{C_1}{w^p}}$ and
\begin{equation*}
\mathcal{R}(t) =e^{-\sum^{t-1}_{s=0}\frac{C_1}{(s+w)^p}}A(0)  + QC_2e^{- \frac{C_1t}{2(t+w)^p}}  \sum_{s=1}^{[t/2]-1} \frac{1}{(s+w)^{p+q}}.
\end{equation*}
Here the second term in the right hand side is assumed to be zero for $1\leq t \leq 3$.
\medskip

\noindent Case 2. If $p=1$, then we have
\begin{equation*}
 A(t) \leq \Big( \frac{w}{t+w}\Big)^{C_1} A(0) + \mathcal{R} (t),
\end{equation*}
where
\begin{equation*}
  \mathcal{R}(t)   = \left\{\begin{array}{ll}  \frac{w^{C_1 -q}}{q-C_1}\cdot\frac{QC_2}{(t+w)^{C_1}}& \textrm{if}~q>C_1
\\
\log \left(\frac{t+w}{w}\right)\cdot\frac{QC_2}{(t+w)^{C_1}} &\textrm{if}~ q=C_1
 \\
\frac{1}{C_1-q}\cdot\left( \frac{w+1}{w}\right)^{C_1}\cdot \frac{QC_2}{(t+w+1)^q}& \textrm{if}~ q<C_1.
\end{array}\right.
\end{equation*}
 
\end{prop}
The proof of Proposition \ref{convergencerate} is given in Appendix \ref{app-a}. 
Now we are ready to prove the remaining of our main results.
\begin{proof}[Proof of Theorems \ref{main} and \ref{thm-2-4}] 
By \eqref{3-12}, we have
\begin{equation}\label{eq-3-10}
\|\mathbf{x}(t) -\bar{\mathbf{x}}(t)\|  \leq \frac{d}{1-\beta} \alpha ([t/2]) + V(t),
\end{equation}
where
\begin{equation}\label{eq-3-50}
\begin{split}
 V(t) &= \beta^t \|\mathbf{x}(0) - \bar{\mathbf{x}}(0)\| + \frac{\beta^{t/2}}{1-\beta} (d \alpha (0)  ).
\end{split}
\end{equation} 
 Inserting this into \eqref{estimate1} we get
\begin{equation}\label{eq-3-11}
 \begin{split}
&\|\bar{\mathbf{x}}(t+1) -\mathbf{x}_*\|\leq \Big(1- {\eta\alpha (t)}\Big) \|\bar{\mathbf{x}}(t) -\mathbf{x}_*\| + L\alpha (t)\cdot\frac{d}{1-\beta}\alpha ([t/2])  + L \alpha (t)V(t).
\end{split}
\end{equation}
Notice that for $w \geq 1$,
\begin{equation*}
 [t/2]+w \geq \frac{t-1}{2}  + w \geq \frac{t+w}{2}
\end{equation*}
for all $t \geq 0$. This implies that 
\begin{equation*}
\alpha([t/2]) =\frac{a}{([t/2] + w)^p} \leq \frac{2^pa}{(t+w)^p}.
\end{equation*}
This, together with \eqref{eq-3-11}, yields the following estimate
\begin{equation*}
 \begin{split}
&\|\bar{\mathbf{x}}(t+1) -\mathbf{x}_*\|\leq \Big( 1- \frac{C}{(t+w)^p}\Big) \|\bar{\mathbf{x}}(t) -\mathbf{x}_*\|  + \frac{C'}{(t+w)^{2p}}  + \frac{aL  V(t)}{(t+w)^p},
\end{split}
\end{equation*} 
where
\begin{equation*}
 C = {\eta a}, \ C' = \frac{Ld 2^p a^2}{1-\beta}.
\end{equation*} 
In order to estimate $\|\bar{\mathbf{x}}(t) - \mathbf{x}_*\|$ from this sequential inequality, we consider two sequences $\{G(t)\}_{t\geq 0}$,  and $\{K(t)\}_{t \geq 0}$ such that
\begin{equation*}
\begin{split}
G(t+1) &= \Big( 1- \frac{C}{(t+w)^p}\Big) G(t) + \frac{C'}{(t+w)^{2p}},\quad \textrm{with}\quad G(0) = \|\bar{\mathbf{x}}(0) - \mathbf{x}_*\|, 
\\
K(t+1)& = \Big( 1- \frac{C}{(t+w)^p}\Big) K(t) + \frac{aLV(t)}{(t+w)^p},\quad \textrm{with}\quad K(0)=0.
\end{split}
\end{equation*}
Then we have the following inequality
\begin{equation}\label{eq-4-21}
\|\bar{\mathbf{x}}(t) -\mathbf{x}_*\| \leq G(t) +   K(t), \ \text{for $t \geq 0$}.
\end{equation}
We first consider the case $p<1$. Using Proposition \ref{convergencerate} we estimate $G(t)$  as
\begin{equation*}
G(t) \leq \delta_1 \cdot\Big( [t/2] + w-1\Big)^{-p} + \mathcal{R}_1(t),
\end{equation*} 
with constants
\begin{equation*}
 \delta_1 = \Big( \frac{w+1}{w}\Big)^{2p} \frac{C'}{C} e^{\frac{C}{w^p}}  =Q_1 \frac{C'}{C} e^{\frac{C}{w^p}}
\end{equation*}
and
\begin{equation*}
\begin{split}
\mathcal{R}_1(t) & = e^{-\sum_{s=0}^{t-1} \frac{C}{(s+w)^p}} G(0) + Q_1C' e^{-\frac{\eta a}{2}\cdot \frac{t}{(t+w)^p}} \sum_{s=1}^{[t/2]-1} \frac{1}{(w+s)^{2p}}\\
&=:Y_1(t) + Y_2(t).
\end{split}
\end{equation*}
Here $Q_1 = \Big( \frac{w+1}{w}\Big)^{2p}$. Next, notice from \eqref{eq-3-50} that
\begin{equation*}
V(t) \leq \frac{\beta^{t/2}}{1-\beta} R_0,
\end{equation*}
where $R_0 = \|\mathbf{x}(0) - \overline{\mathbf{x}}(0)\| + \frac{d}{1-\beta} \alpha (0) $.  Combining this with Lemma \ref{lem-a-2} we find that
\begin{equation*}
K(t) \leq \frac{aLR_0}{(1-\beta)w^p} \bigg( \frac{e^{ -\frac{C}{2 } \frac{(t-1)}{(t+w)^p}}}{1-\sqrt{\beta}} + \frac{\sqrt{\beta}^{[(t-1)/2]}}{1-\sqrt{\beta}}\bigg)=:Y_3(t)
\end{equation*}
Putting these estimates in \eqref{eq-4-21} and observe that
\begin{equation}
\frac{C}{w^p} = \frac{\mu L a}{(\mu +L) w^p} \leq \frac{2\mu L}{(\mu +L)^2} \leq \frac{1}{2}.
\end{equation}
Then we get
\begin{equation*}
\begin{split}
\|\bar{\mathbf{x}}(t) -\mathbf{x}_*\| &\leq \delta \cdot \Big( [t/2] + w-1\Big)^{-p}  +  Y_1(t) + Y_2(t) +Y_3(t)
\\
&\leq \frac{(Q_0 \sqrt{e})a}{1-\beta} \cdot  \Big( [t/2] + w-1\Big)^{-p} +   Y_1(t) + Y_2(t) +Y_3(t),
\end{split}
\end{equation*}
which is the desired estimate.

Next we consider the case $p=1$. In this case, we assume that $a>\frac{2}{\eta}$ which implies $C=\eta a > 2$. Applying Proposition \ref{convergencerate} again to $G(t)$ we deduce
\begin{equation*}
\begin{split}
G(t)& \leq \Big( \frac{w}{t+w}\Big)^{C}\|\mathbf{x}(0) - \mathbf{x}_*\| + \left(\frac{w+1}{w}\right)^C\cdot\frac{Q_1C'}{(C-1)(t+w+1)}
\\
&\leq\Big( \frac{w}{t+w}\Big)^{C}\|\mathbf{x}(0) - \mathbf{x}_*\| + \left(\frac{w+1}{w}\right)^C\frac{2Q_1C'}{C(t+w+1)}.
\end{split}
\end{equation*} 
Here we used the inequality $\frac{1}{C-1} \leq \frac{2}{C}$ since $C>2$. Using $1+x \leq e^x$ and the condition $\frac{a}{w} \leq \frac{2}{\mu + L}$, we have 
\begin{equation*}
\begin{split}
\left(1+\frac{1}{w}\right)^C \leq e^{\frac{\eta a}{w}}\leq  e^{\frac{2\mu L}{(\mu + L)^2}} \leq \sqrt{e}.
\end{split}
\end{equation*}
 We also use Lemma \ref{lem-a-2} to find that 
\begin{equation*}
K(t) \leq \frac{aLR_0}{1-\beta}\left[ \frac{w^l}{1-e^{l/w}\sqrt{\beta}}\cdot \frac{1}{(t-1+w)^{l+1}} + \frac{\sqrt{\beta}^t}{(t-1+w)}\right]=:Y_4(t),
\end{equation*}
where $l$ is any value such that $0 \leq l \leq a-1$ and $e^{(l/w)}\sqrt{\beta} <1$.
Combining these estimates with \eqref{eq-4-21}, we have
\begin{equation*}
\begin{split}
&\|\bar{\mathbf{x}}(t) -\mathbf{x}_*\|  
\\
&\leq \Big( \frac{w}{t+w}\Big)^{C}\|\mathbf{x}(0) - \mathbf{x}_*\| + \frac{2\sqrt{e}Q_1C'}{C(t+w+1)} +Y_4(t)
\\
&= \Big( \frac{w}{t+w}\Big)^{\eta a}\|\mathbf{x}(0) - \mathbf{x}_*\|  + \frac{2\sqrt{e}Q_0}{(1-\beta)}\cdot\frac{a}{(t+w+1)}  +Y_4(t),
\end{split}
\end{equation*}
where $Q_0 = \frac{Q_1Ld2^p}{\eta}$.
This  gives the desired bound. The proof is done.
\end{proof}

\section{Simulation}\label{sec-sim}
In this section, we provide a numerical experiment for the algorithm \eqref{eq-1-1} with decreasing stepsize. We let $m$ be the number of agents and for each $1 \leq i \leq m$, we take $A_i \in \mathbb{R}^{n\times d}$ whose element is chosen randomly following the uniform distribution on $[0,1]$. Next we take a value $x_* \in \mathbb{R}^d$ and define $y_i = A_i x_* + \ep \in \mathbb{R}^n$ where each element of  $\ep \in \mathbb{R}^n$ is chosen from the normal distribution $N(0,0.1)$. The cost is defined as
\begin{equation*}
f(x)  = \frac{1}{m} \sum_{k=1}^m \|A_k x- y_k\|^2.
\end{equation*}
We take a value $Z>0$ and choose $ w=Z^{1/p}$. Next we consider the following values of $a$:
\begin{equation*}
a_1= \frac{w^p}{5(\mu +L)},\quad a_2 = \frac{w^p}{50(\mu +L)},\quad a_3 = \frac{\eta(1-\beta) w^p}{1.1 L(\eta +L)},\quad a_4 = \frac{\eta (1-\beta) w^p}{2L(\eta + L)}.
\end{equation*}
Then $\alpha (0) =\frac{a }{w^p}$ are computed as
\begin{equation*}
 \frac{1}{5(\mu +L)},\quad  \frac{1}{50(\mu +L)},\quad  \frac{\eta (1-\beta)}{1.1L(\eta +L)},\quad \frac{\eta(1-\beta)}{2L(\eta + L)}.
\end{equation*}
We take 
$Z = \frac{(16L(\eta +L))}{(\mu\eta  (1 - \beta))}$. Then, for $a_4$ we have $a (0) = \frac{8}{\mu} > \frac{2}{\mu}$ satisfying the assumption of Theorem \ref{thm-2-4}.
We test the algorithm \eqref{eq-1-1} with $\alpha (t) = \frac{a}{(t+w)^p}$ with above choices of $a$ and $w$ for $p \in \{0.25, 0.5, 0.75, 1\}$. We measure the error $( \sum_{k=1}^m \|x_k (t) -x_*\|^2)^{1/2}$ and the result is presented in Figure \ref{fig1}.
\begin{figure}[h]
	\centering
	\subfloat{{\includegraphics[width=7.5cm]{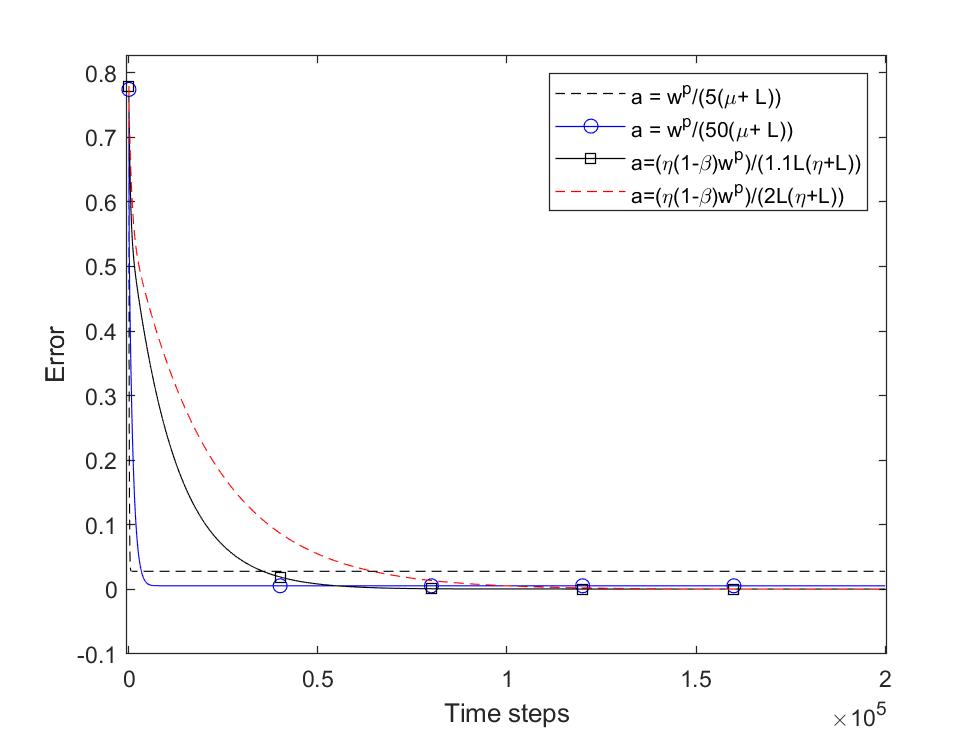} }}%
	\subfloat{{\includegraphics[width=7.5cm]{0.25.jpg} }}%
	\\
	\subfloat{{\includegraphics[width=7.5cm]{0.25.jpg} }}%
	\subfloat{{\includegraphics[width=7.5cm]{0.25.jpg} }}%
	\caption{ (Top-Left) $p=0.25$ (Top-Right) $p=0.5$  (Bottom-left) $p=0.75$ (Bottom-Right) $p=1$}\label{fig1} 
\end{figure}
In the experiment, the constants are computed as follows
\begin{itemize}
\item $L = 17.1256$, $\mu = 0.6961$, $\eta = 0.6689$.
\item $Z = 2.0205 \cdot 10^5 \ (=w^p)$.
\item $a_1 = 2.2675\cdot 10^3$, $a_2 = 226.75$, $a_3 = 20.89$, $a_4 = 11.49$.
\item $\beta = 0.9482$. 
\end{itemize}
As we expected in Theorems \ref{main} and \ref{thm-2-4}, we get fast convergence for large iterations but slow decay in the early stage if the value $a/w^p$ is small and vice versa for large $a/w^p$.

Next we provide a simulation result supporting the result of Lemma \ref{lem-b-1}. We define the cost function $f(x)$  and doubly stochastic $W$ as
$$
f(x)=\frac{f_1(x)+f_2(x)}{2} = \frac{a_1x^2 - a_2x^2}{2}
$$
and 
\begin{equation*}
W = \begin{pmatrix} 1 - \gamma& \gamma \\ \gamma & 1-\gamma \end{pmatrix}.
\end{equation*}
We note that $x=0$ is the optimal value. In this simulation, we take $a_1 = 10$, $a_2 = 6$ and $\gamma = 0.2$ and consider the following values of $k$,
$$
k_1 = 2.1, \ k_2 =2.01, \ k_3 = 2, \ k_4 = 1.99, \ k_5 = 1.9.
$$
We test the algorithm \eqref{eq-1-1} with $\alpha = \frac{\gamma(a_1 + a_2)}{ka_1a_2}$  with above choices of $k$. The initial value $(x_1 (0), x_2 (0))$ is chosen randomly from $(0,50) \times (0,50)$. We measure the quantity $\left(x^2_1(t) + x^2_2(t)\right)^{1/2}$ and the result is presented in Figure \ref{fig2}.

\begin{figure}[h]
	\centering
	\includegraphics[width=7.5cm]{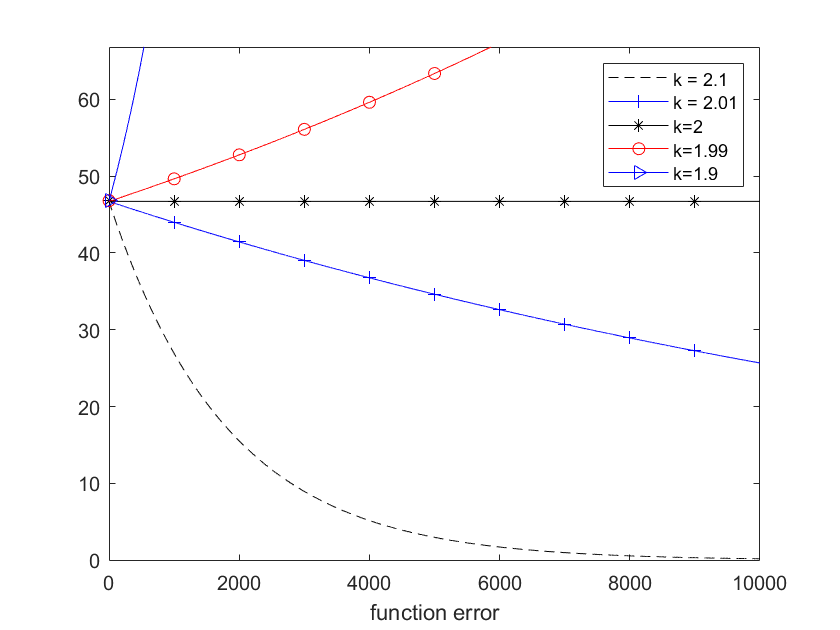} 
	\caption{}\label{fig2}
\end{figure}
As we expected in Lemma \ref{lem-b-1} and Corollary \ref{cor-4-2}, Figure \ref{fig2} shows that the quantities diverge when $\alpha$ is larger than the threshold, $\frac{\gamma(a_1+a_2)}{2a_1a_2}$ and converge when $\alpha$ is smaller than the threshold.

\section{Conclusion}
In this work, we establish the convergence property of the decentralized gradient descent for decreasing stepsize. Different to previous works where each cost function is assumed to be convex, the results of this paper allow each cost function be nonconvex as long as the total cost is assumed stronlgy convex. In addition, we show that the range of the stepsize used in proving the uniform bound is almost sharp. The numerical experiments are provided supporting the results of the paper.

\appendix
\section{Proof of Proposition \ref{convergencerate}}\label{app-a}
This section is devoted to give the proof of  Proposition \ref{convergencerate} and establish Lemma \ref{lem-a-2} which are used in Section \ref{sec5} for obtaining the convergence estimates from the sequential inequalities.
\begin{lem}\label{lem-4-2}\mbox{}
\begin{enumerate}
\item Assume that a continuous function $f: [0,\infty)\rightarrow \mathbb{R}$ is non-increasing on $[0,\infty)$. Then for any integers $b \geq a \geq 0$, we have
\begin{equation*}
\sum_{s=a}^b f(s) \geq \int_a^{b+1} f(v) dv.
\end{equation*}
\item Assume that a continuous function $f:[-1,\infty) \rightarrow \mathbb{R}$ is decreasing on $[0,c)$ and increasing on $[c,\infty)$. Then for any integers $b\geq a \geq 0$, we have
\begin{equation*}
\sum_{s=a}^b f(s) \leq \int_{a-1}^{b+1} f(v) dv.
\end{equation*}
\end{enumerate}
\end{lem}
\begin{proof}
(1) Since $f$ is non-increasing, we have
\begin{equation}\label{eq-4-20}
\begin{split}
\int_a^{b+1} f(s) ds & = \sum_{s=a}^b \int_s^{s+1} f(v) dv 
\\
&\leq \sum_{s=a}^b f(s).
\end{split}
\end{equation}
(2) We consider the case $c \in [a,b]$. Then 
\begin{equation*}
\begin{split}
\sum_{s=a}^b f(s) & = \sum_{s=a}^{[c]} f(s) + \sum_{s=[c]+1}^{b} f(s)
\\
&\leq \sum_{s=a}^{[c]} \int_{s-1}^s f(v) dv  + \sum_{s=[c]+1}^{b} \int_s^{s+1} f(v) dv
\\
&\leq \int_{a-1}^{b+1} f(v) dv,
\end{split}
\end{equation*}
where the first inequality can be proved similarly to  \eqref{eq-4-20}. The case $c \notin [a,b]$ is easier and can be proved similarly. The proof is finished.
\end{proof}
Now we prove Proposition \ref{convergencerate}
\begin{proof}[Proof of Proposition \ref{convergencerate}]
By \eqref{4-4}, we have
\begin{equation}\label{4-14}
A(t) \leq \bigg(1-\frac{C_1}{(t+w-1)^p}\bigg)A(t-1) + \frac{Q C_2}{(t+w)^{p+q}}, \quad \text{for all } t\geq1,
\end{equation}
where 
\begin{equation*}
 Q = \sup_{t \geq 1} \frac{(t+w)^{p+q}}{(t+w-1)^{p+q}} = \Big( \frac{w+1}{w} \Big)^{p+q}.
\end{equation*}
Using \eqref{4-14} iteratively, for $t \geq 1$ we have
\begin{equation}\label{6-1}
\begin{split}
A(t) &\leq \prod^{t-1}_{s=0}\bigg(1-\frac{C_1}{(s+w)^p}\bigg) A(0) + \sum^{t-1}_{s=1}\Bigg[\frac{QC_2}{(s+w)^{p+q}}\prod^{t-1}_{k=s }\bigg(1-\frac{C_1}{(k+w)^p}\bigg)\Bigg] + \frac{QC_2}{(t+w )^{p+q}}\\
&\leq e^{-\sum^{t-1}_{s=0}\frac{C_1}{(s+w)^p}}A(0) + \sum^{t-1}_{s=1} e^{-\sum^{t-1}_{k=s }\frac{C_1}{(k+w)^p}}\frac{QC_2}{(s+w)^{p+q}}  + \frac{QC_2}{(t+w)^{p+q}}.
\end{split}
\end{equation}
Here we used the fact $1-x \leq e^{-x}$ for $x \in \mathbb{R}$. 

We first consider the case $p<1$. Using Lemma \ref{lem-4-2} we note that for any integers $a$ and $b$ with $b \geq a \geq 0$, 
\begin{equation*}
 \begin{split}
\sum_{v=a}^b \frac{C_1}{(v+w)^p}  \geq \int_a^{b+1} \frac{C_1}{(v+w)^p} dv = \frac{C_1}{1-p} \Big( (b+w+1)^{1-p} - (a+w)^{1-p}\Big)
\end{split}
\end{equation*}
since $s \rightarrow \frac{C_1}{(s+w)^p}$ is decreasing for $s \geq 0$.
This gives
\begin{equation*}
e^{-\sum^{t-1}_{k=s } \frac{C_1}{(k+w)^p}} \leq e^{-\frac{C_1((t+w )^{1-p}-(s+w )^{1-p})}{1-p}}.
\end{equation*}
Combining these estimates with \eqref{6-1}, we have
\begin{equation}\label{a-8}
\begin{split}
A(t) &\leq e^{-\sum^{t-1}_{s=0}\frac{C_1}{(s+w)^p}}A(0) + \sum^{t}_{s=1}e^{-\frac{C_1((t+w )^{1-p}-(s+w )^{1-p})}{1-p}}\frac{QC_2}{(s+w)^{p+q}}
\\
&\leq  e^{-\sum^{t-1}_{s=0}\frac{C_1}{(s+w)^p}}A(0)  + J_1 +J_2,
\end{split}
\end{equation}
where
\begin{equation*}
\begin{split}
J_1 &:= \sum^{[t/2]-1}_{s=1}e^{-\frac{C_1((t+w )^{1-p}-(s+w )^{1-p})}{1-p}}\frac{QC_2}{(s+w)^{p+q}} 
\\
J_2 & :=   \sum^{t}_{s=[t/2]}e^{-\frac{C_1((t+w )^{1-p}-(s+w )^{1-p})}{1-p}}\frac{QC_2}{(s+w)^{p+q}}.
\end{split}
\end{equation*}
Here we regard that $J_1 =0$ for $1 \leq t \leq 3$. In fact, the second inequality in \eqref{a-8} is equality except $t=1$. 
We first estimate the second part $J_2$ as follows:
\begin{equation}\label{eq-a-20}
\begin{split}
J_2& = QC_2 e^{-\frac{C_1}{1-p}(t+w )^{1-p}} \sum_{s=[t/2]}^{t} e^{\frac{C_1}{1-p} (s+w)^{1-p}} \frac{1}{(s+w)^{p+q}} 
\\
&\leq QC_2 e^{-\frac{C_1}{1-p}(t+w )^{1-p}} \int_{[t/2]-1}^{t+1 } e^{\frac{C_1}{1-p} (s+w)^{1-p}} \frac{1}{(s+w)^{p+q}} ds
\\
&=: QC_2 e^{-\frac{C_1}{1-p}(t+w )^{1-p}} \cdot I,
\end{split}
\end{equation}
where we used Lemma \ref{lem-4-2} for the  inequality.
Notice that 
\begin{equation*}
 \left( \frac{1}{C_1} e^{\frac{C_1}{1-p} (s+w)^{1-p}} \right)' =  e^{\frac{C_1}{1-p} (s+w)^{1-p}}\cdot\frac{1}{(s+w)^p}.
\end{equation*}
Using this and an integration by parts, we get
\begin{equation}\label{4-25}
\begin{split}
I & = \frac{1}{C_1}\left[  e^{\frac{C_1}{1-p}(s+w)^{1-p}}\cdot \frac{1}{(s+w)^q} \right]_{[t/2]-1}^{t+1} + \frac{q}{C_1} \int_{[t/2]-1}^{t+1}  e^{\frac{C_1}{1-p} (s+w)^{1-p}} \frac{1}{(s+w)^{q+1}} ds
\\
& = \frac{1}{C_1}\left[ e^{\frac{C_1}{1-p} (t+1+w )^{1-p}} (t+1+w )^{-q} -e^{\frac{C_1}{1-p}([t/2]+w-1)^{1-p}} ([t/2]+w-1)^{-q} \right]
\\
&\quad + \frac{q}{C_1} \int_{[t/2]-1}^{t+1} e^{\frac{C_1}{1-p} (s+w)^{1-p}} \frac{1}{(s+w)^{q+1}} ds.
\end{split}
\end{equation}
Note that
\begin{equation*}
 \frac{q}{C_1} \int_{[t/2]-1}^{t +1} e^{\frac{C_1}{1-p} (s+w)^{1-p}} \frac{1}{(s+w)^{q+1}} ds \leq \frac{q}{C_1} e^{\frac{C_1}{1-p} (t+1+w )^{1-p}} \int_{[t/2]-1}^{t +1} \frac{1}{(s+w)^{q+1}} ds
\end{equation*}
and
\begin{equation*}
 \int_{[t/2]-1}^{t+1 } \frac{1}{(s+w)^{q+1}} ds =  -\frac{(t+1+w )^{-q} - ([t/2]+w-1)^{-q}}{q}.
\end{equation*} 
Using these estimates, the integration part of $I$ in \eqref{4-25} is bounded by
\begin{equation*}
 -\frac{1}{C_1} e^{\frac{C_1}{1-p} (t+1+w)^{1-p}} \Big[ (t+1+w )^{-q} - ([t/2]+w-1)^{-q}\Big].
\end{equation*}
Combining this with \eqref{eq-a-20}, we deduce
\begin{equation*}
  \begin{split}
J_2 & \leq \frac{QC_2}{C_1} e^{-\frac{C_1}{1-p} (t+w )^{1-p}}
\\
&\quad \cdot \Bigg[ \left( e^{\frac{C_1}{1-p} (t+1+w )^{1-p}} (t+1+w )^{-q} - e^{\frac{C_1}{1-p} ([t/2]+w-1)^{1-p}} ([t/2]+w-1)^{-q}\right) 
\\
&\quad - e^{\frac{C_1}{1-p} (t+1+w)^{1-p}} \left( (t+1+w )^{-q} - ([t/2]+w-1)^{-q}\right)\Bigg]
\\
&\leq \frac{QC_2}{C_1} e^{ - \frac{C_1}{1-p}\big[ (t+w)^{1-p} - (t+1+w)^{1-p}\big]} ([t/2]+w-1)^{-q} 
\\
&\leq \frac{QC_2}{C_1} e^{{\frac{C_1}{w^{p}}}}([t/2]+w-1)^{-q} ,
\end{split}
\end{equation*}
where we used that $(t+w+1)^{1-p} - (t+w)^{1-p} \leq \frac{1-p}{(t+w)^p}\leq \frac{1-p}{w^p}$ in the last inequality. Next, using that $s\leq[t/2]-1$ in the summation of $J_1$, we derive the following inequality:
\begin{equation*}
\begin{split}
J_1&\leq QC_2  e^{-\frac{C_1}{1-p}[(t+w )^{1-p} - ([t/2]-1+w)^{1-p}]} \sum_{s=1}^{[t/2]-1} \frac{1}{(s+w)^{p+q}}
\\
&\leq  QC_2   e^{-\frac{C_1 t}{2(t+w)^p}}\sum_{s=1}^{[t/2]-1} \frac{1}{(s+w)^{p+q}}.
\end{split}
\end{equation*}
Here we used  $(t+w)^{1-p} - ([t/2]-1+w)^{1-p} \geq (1-p)\frac{t-[t/2] + 1}{(t+w)^p}\geq \frac{(1-p)t}{2(t+w)^p}$. Combining the above two estimates on $J_1$ and $J_2$ in \eqref{a-8}, we deduce
\begin{equation*}
A(t) \leq  \frac{QC_2}{C_1}e^{\frac{C_1}{w^p}} ([t/2]+w-1)^{-q}+ \mathcal{R}(t),
\end{equation*}
where 
$$
\mathcal{R}(t) =e^{-\sum^{t-1}_{s=0}\frac{C_1}{(s+w)^p}}A(0)  + QC_2 e^{- \frac{C_1t}{2(t+w)^p}} \sum_{s=1}^{[t/2]-1} \frac{1}{(w+s)^{p+q}}.
$$
This establish the proposition for $p \in (0,1)$.
     
Next we consider the case $p=1$. As in  \eqref{6-1} we have
\begin{equation}\label{eq-a-21}
\begin{split}
A(t) \leq e^{-\sum^{t-1}_{s=0}\frac{C_1}{s+w}}A(0) + \sum^{t-1}_{s=1} e^{-\sum^{t-1}_{k=s }\frac{C_1}{k+w}}\frac{QC_2}{(s+w)^{1+q}}  + \frac{QC_2}{(t+w)^{1+q}}.
\end{split}
\end{equation}
 Notice that for any integers $a$ and $b$ with $b \geq a \geq 0$, we use Lemma \ref{lem-4-2} to have
\begin{equation*}
\begin{split}
 \sum_{s=a}^{b} \frac{C_1}{s+w} \geq \int_a^{b+1} \frac{C_1}{s+w} ds &=  \log \bigg(\frac{b+w+1}{a+w}\bigg)^{C_1}.
\end{split}
\end{equation*}
Using this in \eqref{eq-a-21} we get
\begin{equation}\label{4-30}
\begin{split}
A(t) &\leq \bigg(\frac{w}{t+w}\bigg)^{C_1}A(0) + \sum^{t}_{s=1}\bigg(\frac{s+w}{t+w}\bigg)^{C_1}\frac{QC_2}{(s+w)^{1+q}}\\
&\leq \bigg(\frac{w}{t+w}\bigg)^{C_1}A(0) + \bigg(\frac{1}{t+w}\bigg)^{C_1}\sum^{t}_{s=1}\frac{QC_2}{(s+w)^{1+q-C_1}}.
\end{split}
\end{equation}


\noindent Case 1. Suppose that $1+q -C_1 \neq 1$. Then we have 
\begin{equation*}
 \begin{split}
\sum_{s=1}^{t} \frac{1}{(s+w)^{1+q-C}}& \leq \int_{0}^{t+1} \frac{1}{(s+w)^{1 + q-C_1}} ds 
\\
& = \frac{1}{q-C_1} \Big[ w^{C_1-q} - (t+1+w)^{C_1-q}\Big],
\end{split}
\end{equation*}
where we used Lemma \ref{lem-4-2} for the first inequality. Hence $A(t)$ is bounded by 
\begin{equation*}
A(t) \leq \bigg(\frac{w}{t+w}\bigg)^{C_1}A(0) + \frac{w^{C_1 -q}}{q-C_1}\frac{QC_2}{(t+w)^{C_1}} -\frac{1}{q-C_1}\frac{QC_2}{(t+w+1)^q}\frac{(t+w+1)^{C_1}}{(t+w)^{C_1}}.
\end{equation*}

\noindent Case 2. Suppose that $1+q-C_1 =1$. Then we have, 
\begin{equation*}
 \begin{split}
\sum_{s=1}^{t} \frac{1}{s+w} \leq \int_{0}^{t} \frac{1}{s+w} ds = \log \bigg(\frac{t+w}{w}\bigg).
\end{split}
\end{equation*}
Hence $A(t)$ is bounded by
\begin{equation*}
A(t) \leq \bigg(\frac{w}{t+w}\bigg)^{C_1}A(0) +  \log \bigg(\frac{t+w}{w}\bigg)\frac{QC_2}{(t+w)^{C_1}}.
\end{equation*}
Combining the above estimates, we find
\begin{equation*}
 A(t) \leq \Big( \frac{w}{t+w}\Big)^{C_1} A(0) +  \mathcal{R}(t) ,
\end{equation*}
where
\begin{equation*}
  \mathcal{R}(t)   = \left\{\begin{array}{ll}  \frac{w^{C_1 -q}}{q-C_1}\cdot\frac{QC_2}{(t+w)^{C_1}}& \textrm{if}~q>C_1
\\
\log \left(\frac{t+w}{w}\right)\cdot\frac{QC_2}{(t+w)^{C_1}} &\textrm{if}~ q=C_1
 \\
\frac{1}{C_1-q}\cdot\left( \frac{w+1}{w}\right)^{C_1} \cdot\frac{QC_2}{(t+w+1)^q}& \textrm{if}~ q<C_1.
\end{array}\right.
\end{equation*}
The proof is done.
\end{proof}

\begin{rem}

\textbf{(Bound of $\mathcal{R}(t)$)} Here we show that for $0<p<1$, the $\mathcal{R}(t)$ of Proposition \ref{convergencerate} satisfies $\mathcal{R}(t) = O(t^{-N})$. Recall that
$$\mathcal{R}(t)=e^{-\sum^{t-1}_{s=0}\frac{C_1}{(s+w)^p}}A(0)  + QC_2 e^{- \frac{C_1t}{2(t+w)^p}} \sum_{s=1}^{[t/2]-1} \frac{1}{(s+w)^{p+q}}.$$ 
Since we know that $ QC_2 e^{- \frac{C_1t}{2(t+w)^p}} \sum_{s=1}^{[t/2]-1} \frac{1}{(s+w)^{p+q}} =O(t^{-N}) $, it is sufficient to show that $e^{-\sum^{t-1}_{s=0}\frac{C_1}{(s+w)^p}}A(0)= O(t^{-N})$. 
Note that
\begin{equation*}
 \begin{split}
\sum_{s=0}^{t-1} \frac{C_1}{(s+w)^{p}} & \geq \int_0^{t} \frac{C_1}{(s+w)^p} ds
\\
& = \frac{C_1}{1-p} \Big[ (t+w)^{1-p} -w^{1-p}\Big].
\end{split}
\end{equation*}
Using this we estimate the first term of $\mathcal{R}(t)$ as
\begin{equation}\label{eq-5-1}
e^{-\sum_{s=0}^{t-1} \frac{C_1}{(s+w)^p}} A(0) \leq A(0) e^{\frac{C_1}{1-p}w^{1-p}} e^{-\frac{C_1}{1-p}(t+w)^{1-p}}=O(t^{-N}).
\end{equation}

\end{rem}
\begin{lem}\label{lem-a-2}Fix $p \in (0,1]$. Let $a, b$ and $w$ be any positive values and $\beta \in (0,1)$ satisfying $\frac{a}{w^p} <1$. 
Assume that  a positive sequence $\{B(t)\}_{t \geq 0}$ satisfies
\begin{equation*}
B(t+1) \leq \Big(1 - \frac{a}{(t+w)^p}\Big) B(t) + \frac{b \beta^t}{(t+w)^p}
\end{equation*}
for $t \geq 0$ and $B(0)=0$. Then we have the following estimates.
\begin{enumerate}
\item If $p=1$, then
\begin{equation*}
B(t+1) \leq J_{a,w,\beta} \cdot \frac{bw^{a-1}}{(t+w)^a}  + \frac{b \beta^t}{(t+w)}, 
\end{equation*} 
where $J_{a,w,\beta} = \sum_{j=0}^{\infty} \frac{(j+w)^{a-1}}{w^{a-1}} \beta^j$. Or we have
\begin{equation*}
B(t+1) \leq   \frac{bw(w+1)^{l-1}}{(t+1+w)^{l+1}}\cdot \frac{1}{1-  \beta e^{l/(w+1)}} + \frac{b\beta^t}{(t+w)},
\end{equation*}
where $l$ is any number satisfying $0 \leq l \leq a-1$ and $ \beta e^{l/(w+1)} <1$.
\item If $p \in (0,1)$, then
\begin{equation*}
B(t+1) \leq \frac{b}{ w^p (1-\beta)} \Big( e^{ -\frac{a}{2}\cdot \frac{t}{(t+1+w)^p} }+  {\beta^{[t/2]}}\Big).
\end{equation*}
\end{enumerate}
\end{lem}
\begin{proof}
Letting $a(t) = 1-\frac{a}{(t+w)^p}$ and $b(t) = \frac{b \beta^t}{(t+w)^p}$ for simplicity, we find that
\begin{equation}\label{eq-a-22}
\begin{split}
B(t+1)& \leq \Big(\prod_{k=0}^t a(k)\Big) B(0) + \sum_{j=0}^{t-1} \Big(\prod_{k=j+1}^t a(k)\Big) b(j) + b(t)
\\
&= \sum_{j=0}^{t-1} \Big(\prod_{k=j+1}^t a(k)\Big) b(j) + b(t).
\end{split}
\end{equation}
(Case $p=1$). In this case, we have
\begin{equation*}
B(t+1) \leq \sum_{j=0}^{t-1} \Big(\prod_{k=j+1}^t \Big( 1- \frac{a}{k+w}\Big) \Big) \frac{ b \beta^j}{(j+w)} + \frac{ b \beta^t}{(t+w)}.
\end{equation*}
Notice that
\begin{equation*}
\begin{split}
\prod_{k=j+1}^t \Big( 1 -\frac{a}{k+w}\Big) & \leq e^{ - \sum_{k=j+1}^t \frac{a}{k+w}}
\\
&\leq e^{- a \int_{j+w+1}^{t+w+1} \frac{1}{s} ds }
\\
&  = e^{ -a \Big( \log (t+w+1) - \log (j+w+1)\Big)}
\\
& = \Big( \frac{j+1+w}{t+1+w}\Big)^a.
\end{split}
\end{equation*}
Using this we get
\begin{equation*}
\begin{split}
B(t+1)& \leq \sum_{j=0}^{t-1} \frac{(j+1+w)^{a}}{(t+1+w)^a} \frac{b \beta^j}{j+w} + \frac{ b\beta^t}{(t+w)}
\\
& =\Big(\frac{w+1}{w}\Big) \sum_{j=0}^{t-1} \frac{(j+1+w)^{a-1}}{(t+1+w)^{a}} b \beta^j + \frac{b \beta^t}{(t+w)}
\\
& \leq J_{a,w,\beta} \cdot \frac{bw^{a-1}}{(t+1+w)^a} + \frac{b \beta^t}{(t+w)},
\end{split}
\end{equation*}
where $J_{a, w, \beta} = \big( \frac{w+1}{w}\big)\sum_{j=0}^{\infty} \frac{(j+1+w)^{a-1}}{w^{a-1}} \beta^j$.

For any $0 \leq l \leq a-1$, we may also bound it as
\begin{equation*}
\begin{split}
B(t+1) & \leq \Big(\frac{w}{w+1}\Big) \frac{b(w+1)^l}{(t+1+w)^{l+1}}\sum_{j=0}^{t-1} \Big(1 + \frac{j}{w+1}\Big)^l \beta^j
 + \frac{ b\beta^t}{(t+w)}
\\
&\leq  \Big(\frac{w}{w+1}\Big) \frac{b(w+1)^{l}}{(t+1+w)^{l+1}} \sum_{j=0}^{t-1}  e^{j l/(w+1)} \beta^j + \frac{ b\beta^t}{(t+w)}
\\
&\leq  \Big(\frac{w}{w+1}\Big) \frac{b(w+1)^l}{(t+1+w)^{l+1}}\cdot \frac{1}{1-  \beta e^{l/(w+1)}} + \frac{b\beta^t}{(t+w)},
\end{split}
\end{equation*}
where we used that $1+x \leq e^x$ for $x \geq 0$ in the second inequality.
\medskip 

\noindent 
(Case $p\in (0,1)$). We estimate \eqref{eq-a-22} further as
\begin{equation*}
\begin{split}
B(t+1)& \leq \sum_{j=0}^{[t/2]-1} \Big( \prod_{k=j+1}^t a(k)\Big) b(j) + \sum_{j=[t/2]}^{t-1} \Big( \prod_{k=j+1}^t a(k)\Big) b(j) + b(t)
\\
&\leq \Big( \prod_{k=[t/2]}^t a(k)\Big) \sum_{j=0}^{[t/2]-1} b(j) + \sum_{j=[t/2]}^t b(j).
\end{split}
\end{equation*}
Since $b(t) = \frac{b\beta^t}{(t+w)^p}$, we have $b(t) \leq \frac{b}{w^p  }\beta^t$. Thus, 
\begin{equation*}
\begin{split}
B(t+1) &\leq \frac{1}{1-\beta}\Big( \prod_{k=[t/2]}^t a(k)\Big) \frac{b}{w^p } + \frac{b}{w^p} \sum_{j=[t/2]}^t \beta^j
\\
&\leq \frac{b}{(1-\beta) w^p   } \Big( \prod_{k=[t/2]}^t a(k)\Big) + \frac{b}{w^p  (1-\beta)} \beta^{[t/2]}.
\end{split}
\end{equation*}
We estimate
\begin{equation*}
\begin{split}
\prod_{k=[t/2]}^{t} \Big( 1- \frac{a}{(k+w)^p}\Big) & \leq e^{- \sum_{k=[t/2]}^{t} \frac{a}{(k+w)^p}}
\\
&\leq e^{ -a \int_{[t/2]+w}^{t+w+1} \frac{1}{s^p} ds}
\\
& = e^{ -\frac{a}{ (1-p)} \big( (t+w+1)^{1-p} - ([t/2]+w)^{1-p}\big)}.
\end{split}
\end{equation*}
Therefore we have
\begin{equation*}
\begin{split}
B(t+1) &\leq  \frac{b}{  (1-\beta) w^p} \bigg( e^{ -\frac{a}{ (1-p)} ( (t+w+1)^{1-p} - ([t/2]+w)^{1-p}) } + {\beta^{[t/2]}}\bigg)\\
&\leq \frac{b}{(1-\beta) w^p} \Big({e^{ -\frac{a}{2}\cdot \frac{t}{(t+w+1)^p} }}+ {\beta^{[t/2]}}\Big).
\end{split}
\end{equation*}
Here we used  $(t+1+w)^{1-p} - ([t/2]+w)^{1-p} \geq (1-p)\frac{t-[t/2] + 1}{(t+1+w)^p}\geq \frac{(1-p)t}{2(t+1+w)^p}$. The proof is done.
\end{proof}

\end{document}